\documentclass{article}
\usepackage{graphicx,epsfig,psfrag,rotate,xcolor}

\usepackage{amsmath,amsfonts,amssymb,latexsym,amsthm}
\usepackage{setspace,enumerate,ifthen,subfig}
\usepackage{hyperref}
\usepackage{cite}
\usepackage{ifpdf}
\usepackage{subfiles}
\usepackage{enumitem}
\usepackage{mathrsfs}
\usepackage{enumerate}

\newcommand{\R}{\mathbb{R}}
\newcommand{\N}{\mathbb{N}}

\newcommand{\prob}{\mathbb{P}}

\def\defeq{\mathrel{\mathop:}=}

\newtheorem{thm}{Theorem}

\newtheorem{lem}{Lemma}
\newtheorem{eg}{Example}
\newtheorem{rem}{Remark}
\newtheorem{prp}{Proposition}

\newcommand*{\Scale}[2][4]{\scalebox{#1}{$#2$}}%

\begin{document}

\title{\LARGE \bf On Classical Control and Smart Cities}

\author{Andr\'e~R.~Fioravanti, Jakub~Mare\v{c}ek, Robert~N.~Shorten, \\ Matheus~Souza, Fabian~R.~Wirth
\thanks{
A. R. Fioravanti and M. Souza are at the University of Campinas, Brazil.
J. Marecek is with IBM Research -- Ireland, in Dublin, Ireland.
R. N. Shorten is at the University College Dublin, Dublin, Ireland.
F. Wirth is at the University of Passau, Germany.
This work
 was in part supported by Science Foundation Ireland grant 11/PI/1177
 and received funding from the European Union Horizon 2020 Programme (Horizon2020/2014-2020), under grant agreement no 68838.
}}%

\maketitle
\thispagestyle{empty}
\pagestyle{empty}

\begin{abstract}
We discuss the applicability of classical control theory to problems in smart grids and smart cities.
We use tools from iterated function systems to identify controllers with desirable properties.
In particular, controllers are identified that can be used to design not only stable closed-loop systems, but also
 to regulate large-scale populations of agents in a predictable manner.
We also illustrate by means of an example and associated theory that many classical controllers are not be suitable for deployment in these applications.
\end{abstract}

\section{Introduction}

Smart Grid and Smart City research, at a very high level, is about making the best use of existing resources, as we try to manage meet demand for them under application-specific constraints.
This
is a classical consideration of Control Theory, and
while classical control has much to offer in such application areas, 
there are aspects of these contemporary applications 
 that also offer the opportunity for practitioners to explore new boundaries in control \cite{7563957}.

 First, classical control is typically concerned with regulating a single system such that the system behaviour achieves a desired behaviour in an optimal way. In contrast, in Smart Grid and Smart City applications, the aggregate effect of the actions of an ensemble of (often human) agents is a variable of considerable interest. 
 Further, classical control is concerned with the control of systems whose structure does not vary over time. 
 On the other hand, in Smart Grid and Smart City applications, we typically wish to control and influence the behaviour of large-scale populations, where the number of agents in the ensemble may be uncertain and varying over time.
Third, data sets involved are often obtained in a closed loop setting. That is, operators' decisions are often reflected in the data sets.
Finally, a fundamental difference between classical control and Smart Grid and Smart City control, is the need to study the effect of the control signals on the statistical properties of the populations that we wish to influence. 
Among all of these fundamental differences, it is this last issue, of the need of ergodic feedback systems, that is perhaps most alien to the classical control theorist, and yet the issue that is perhaps the most pressing in real-life applications since the predictability, at the level of individual agents, underpins operators' ability to issue contracts. \newline

In this paper, our objective is to discuss and explain the need for ergodic control design, present some initial results that identify controllers that give rise to ergodic feedback systems, and also identify some classical controllers,  which may give rise to difficulties.
Our starting point is the observation that many problems that are considered in Smart Grids and Smart Cities can be cast in a framework, where a  large number of agents, such as people, cars, or machines, often with unknown objectives, compete for a limited resource. The challenge of allocating these resources in a manner that is not wasteful, which gives an optimal return on the use of these resources for society, and which, in addition, gives a guaranteed level of service to each of the agents competing for that resource, gives rise to a whole host of problems, which are best addressed in a control-theoretic manner. From a practical perspective, some of these problems may seem unrelated to traditional applications of control.
For example, balancing supply and demand in an energy system may seem familiar to a control engineer, at a high level, there are many variants and contract types, which introduce much additional complexity. Allocating parking spaces, regulating cars competing for shared road space, allocating shared bikes, all at the same time guaranteeing fair and equal access to the participants, are examples of Smart City type applications that seem more removed from the traditional interests of control engineering. However, while each of these applications seem very different, certain key features remain the same. Resource utilisation should be maximised while delivering a certain quality of service to individual agents. From the perspective of a control engineer, this latter statement decompose into three objectives; two of which are familiar to the classical control engineer, one of which constitutes a relatively new consideration. Our first objective is to fully utilise the resource. From a control engineers' perspective, this is a regulation problem. We would then like to make optimal use of the resource. While both of these objectives are concerned with affecting the aggregate behaviour of an agent population, they make no attempt to control the manner in which the agents orchestrate their behaviour to achieve this aggregate effect. Our final, and third, consideration thus focuses on the effects of the control on the microscopic properties of the agent population.\newline

Ultimately, this third concern focuses on the stochastic process that governs the share of the resource that is allocated to an {\em individual agent}. For example, we may wish that each agent, on average, receives a fair share of the resource over time, or, at a much more fundamental level, we wish the average allocation of the resource to each agent over time to be a stable quantity that is entirely predictable and
which does not depend on initial conditions, and which is not sensitive to noise entering the system. From the point of view of the design of the feedback system, these latter concerns are related to the existence of the {\em unique invariant measure} that governs the distribution of the resource amongst the agents in the long run.
Thus, the design of feedback systems for deployment in cities must consider not only the traditional notions of regulation and optimisation, but also the guarantees concerning
the existence of this unique invariant measure. As we shall see, this is not a trivial task and many familiar control strategies, in very simple situations,
do not necessarily give rise to feedback systems, which posses all three of these features.

\section{The Problem}
\label{sec:problem}

\ifpdf

\begin{figure}
\centering
\includegraphics[width=\columnwidth,clip=true]{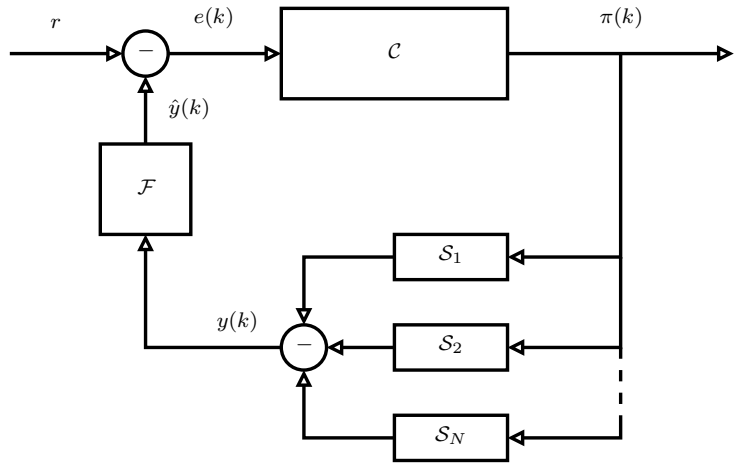}
\caption{Feedback model.}\label{system}
\end{figure}

\else
\begin{figure}
  \centering
  \begin{psfrags}
    \psfrag{+}[c]{\footnotesize $-$}
    \psfrag{r}[b]{{\footnotesize $r$}}
    \psfrag{e}[b]{{\footnotesize $e(k)$}}
    \psfrag{pi}[b]{{\footnotesize $\pi(k)$}}
    \psfrag{xb}[bl]{{\footnotesize $\hat y(k)$}}
    \psfrag{y}[c]{{\footnotesize $y(k)$}}
    \psfrag{C}[c]{\footnotesize $\mathcal{C}$}
    \psfrag{F}[c]{\footnotesize $\mathcal{F}$}
    \psfrag{x1}[c]{\footnotesize $\mathcal{S}_1$}
    \psfrag{x2}[c]{\footnotesize $\mathcal{S}_2$}
    \psfrag{xn}[c]{\footnotesize $\mathcal{S}_N$}
  \includegraphics[width=0.8\columnwidth]{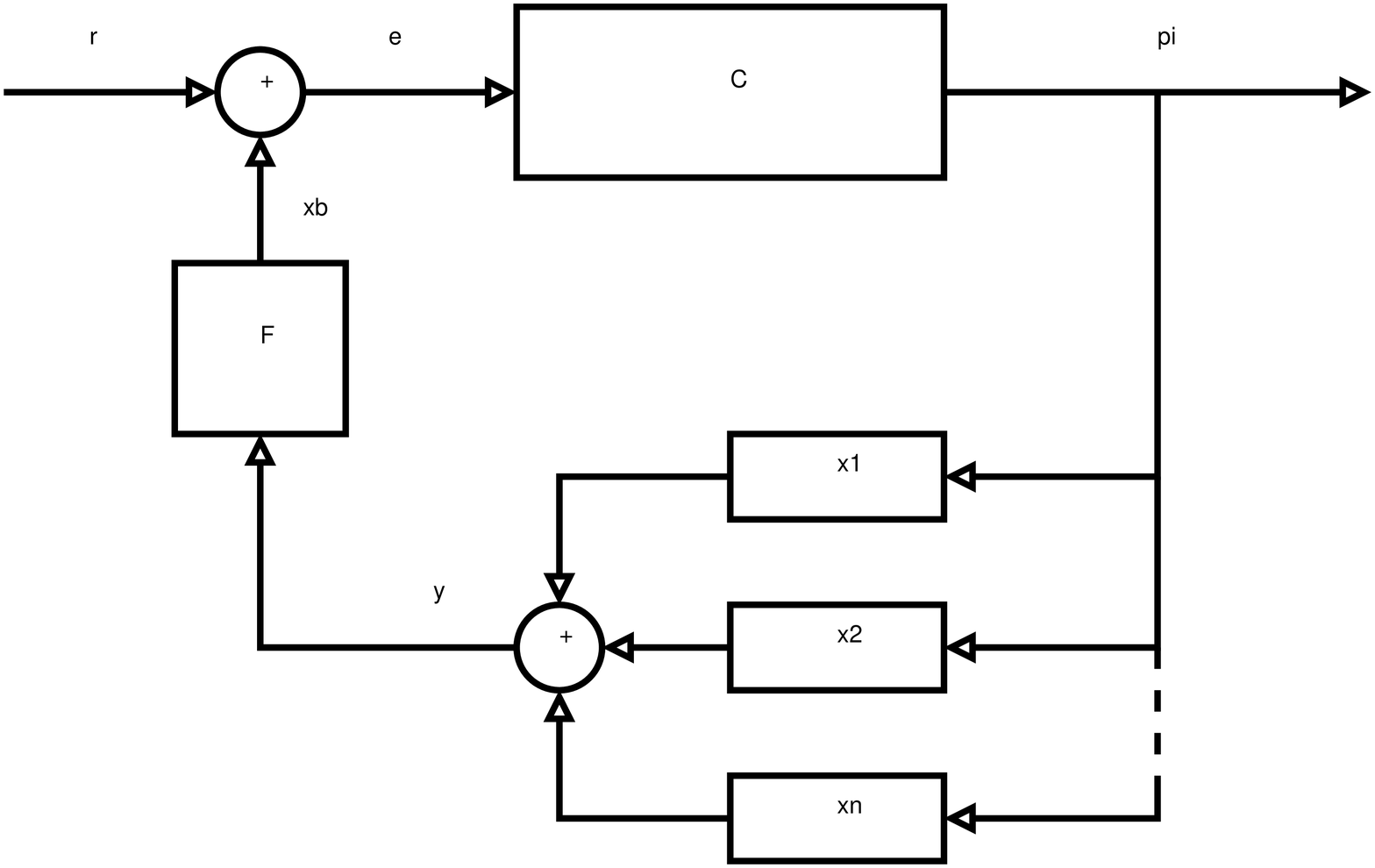}
  \caption{Feedback model.}\label{system}
  \end{psfrags}
\end{figure}
\fi

We consider a closed loop displayed in Figure \ref{system},
comprising a controller, a number of agents, and a filter, in discrete time.
A controller  $\mathcal{C}$
produces a signal $\pi(k)$ at time $k$.
In response, $N \in \N$ agents
modelled by systems $\mathcal{S}_1$, $\mathcal{S}_2$, \ldots, $\mathcal{S}_N$
amend their use of the resource.
We model the use $x_i(k)$ of agent $i$ at time $k$ as a random variable,
where the randomness can be a result of the inherent randomness in the
reaction of user $i$ to the control signal $\pi(k)$, or the response to a
control signal that is intentionally randomized, or the
receipt of the control signal subject to random perturbations.
The aggregate resource utilisation $y(k) = \sum_{i = 1}^N x_i(k)$ at time $k$
is then also a random variable.
The controller does not have access to either $x_i(k)$ or $y(k)$,
but only to
the error signal $e(k)$, which is the difference of
$\hat y(k)$, the output of a filter $\mathcal{F}$,
and $r$, the desired value of $y(k)$. Even though this setup
is elementary, it represents many applications found in Smart Grid and Smart City applications.

Let us consider two simple examples of the structures
${\mathcal C} \, : \, (e, x_c) \mapsto \pi$ and
${\mathcal F} \, : \, (y, x_f) \mapsto \hat y$,
where $x_c \in X_C \subseteq \R^{n_c}$ is the internal state of the controller in dimension $n_c$,
      $x_f \in X_F \subseteq \R^{n_f}$ is the internal state of the filter in dimension $n_f$,
      and $n_c, n_f$ are non-negative integers:
\begin{eg}[Linear setting]
    \label{linear}
In the linear setting, the controller dynamics may be:
\begin{equation} \label{eq_c}
{\mathcal C} ~:~ \left\{ \begin{array}{rcl}
x_c(k+1) & = & A^c x_c(k) + B^c e(k), \vspace{0.1cm} \\
\pi(k) & = & C^c x_c(k) + D^c e(k),
\end{array} \right.
\end{equation}
where $x_c \, : \, \N \to \R^{n_c}$ is the internal state of the controller in dimension $n_c$.
One could adopt a linear model for the filter ${\mathcal F}$, based on the classic IIR/FIR structures \cite{Oppenheim}.
Remembering that $y$ is a linear combination of the states, one has
\begin{equation}\label{eq_f}
{\mathcal F} ~:~ \left\{ \begin{array}{rcl}
x_f(k+1) & = & A^f x_f(k) + B^f y(k) + C^f \tilde y(k), \vspace{0.1cm} \\
\hat y(k) & = & D^f x_f(k),
\end{array} \right.
\end{equation}
where $\tilde y$ stores the previous $M$ values of $y$; that is, $\tilde y$ evolves by $\tilde y(k+1) = J y(k) + L \tilde y(k)$ with
\begin{equation}
\label{JL}
J = \left[ \begin{array}{c} 1  \\ 0 \\  \vspace{-0.12cm} 0 \\ \vdots \\ 0 \end{array} \right], \quad L = \left[ \begin{array}{ccccc} 0 & 0 & \cdots & 0 & 0  \\ 1 & 0 & \cdots & 0 & 0 \\\vspace{-0.12cm}  0 & 1 & \cdots & 0 & 0 \\ \vdots & \vdots & \ddots & \vdots & \vdots\\ 0 & 0 & \cdots & 1 & 0 \end{array} \right].
\end{equation}
\end{eg}

As a very special case, one could consider:

\begin{eg}
    \label{linear2}
A Proportional-Integral (PI) controller: 
\begin{equation}\label{pid1}
{\mathcal C} ~:~ \pi(k) = \pi(k-1) + \kappa \big[ e(k) - \alpha e(k-1) \big],
\end{equation}
with a state $x_c$ modelling one-element buffers for $e$ and $\pi$, and a moving-average filter:
\begin{equation} \label{fir}
{\mathcal F} ~:~
\hat y(k) = \frac{y(k) + y(k-1)}{2},
\end{equation}
with state $x_f$ modelling one-element buffer for $y$.
\end{eg}

Loosely speaking, our aim is for the controller ${\mathcal C}$ and filter ${\mathcal F}$ to make the long-run behaviour of the system acceptable,
for all possible distributions of the initial state, both from the perspective of the individual agents, and from the perspective of the macroscopic behaviour of the network of agents.
Somewhat more formally, we aim to find conditions such that for every stable linear controller $\mathcal{C}$ and every stable linear filter $\mathcal{F}$, the feedback loop regulates the aggregate resource use, and converges in distribution to a unique invariant measure.
The existence of a unique invariant measure may mean many things in practice: for example that agents get a fair share of the resource, or even that their time-averaged share is predictable in some sense. For example within Smart Grids, many demand-response management (DRM) schemes \cite{Hiskens2011}
consider interruptible loads (IL) \cite{192993}, such as in heating, ventilation, and air conditioning (HVAC) systems \cite{Alagoz2013,Salsbury2013}.
Specifically, one aims to regulate the total active power demand $r$ of the ILs,
while $x_i(k)$ determines the active power supplied to IL $i$ at time $k$.
Although continuous changes to the demand are possible, in theory,
many current schemes consider this binary notion of interruptibility \cite{192993}, in practice.
Then, it may desirable for an IL to have the same rate of interruptions as any other
IL with the same costs, independent of whether the IL was on or off at the start.
Many researchers, e.g. \cite{5560848,Alagoz2013,Salsbury2013}, propose to use the classic
Proportional-Integral-Derivative (PID)  
controllers in this setting. As we shall see, such controllers do not
guarantee the existence of the unique invariant measure
and it is not difficult to construct situations in which they fail to provide
agent-specific outcomes independent of the agent's initial state.

\section{Notation and Preliminaries}

\subsection{Markov Chains}



Let $\Sigma$ be a closed subset of $\R^n$ with the usual Borel
$\sigma$-algebra ${\cal B}(\Sigma)$. We call the elements of ${\cal B}$
events. A Markov chain on $\Sigma$ is a sequence
of ($\Sigma$-valued) random
vectors $\{ X(k)\}_{k\in\N}$ with the Markov property, that is the
probability of an event conditioned on past events is given by conditioning on the previous event, \emph{i.e.}, we always have
\begin{multline*}
    \prob(X(k+1) \in G \,|\, X(j)=x_{j}, \,  j=0, 1, \dots, k) \\ = \prob\big(X(k+1)\in G\,|\, X(k)=x_{k}\big),
\end{multline*}
where $G$ is an event
and $k \in \N$. We assume the Markov chain is time-homogeneous and the transition operator $P$ of the Markov chain is defined
by
\begin{equation*}
    P(x,G) := \prob(X(k+1) \in G \vert X(k) = x).
\end{equation*}
 If $X_0$ is distributed according to an initial distribution
$\lambda$ we denote by $\prob_\lambda$ the probability measure induced on
the space of sequences with values in $\Sigma$. Conditioned on an
initial distribution $\lambda$, the random variable $X(k)$ is distributed
according to the measure $\lambda_k$ which is given by
\begin{equation}
    \label{eq:measureiteration}
    \lambda_{k+1}(G) :=  \lambda_{k}P(G)   := \int_{\Sigma} P(x, G) \, \lambda_k(d x), \quad G
    \in {\cal B}.
\end{equation}
A measure $\mu$ on $\Sigma$ is
called invariant with respect to the Markov process $\{ X(k) \}$ if it is
a fixed point for the iteration described by \eqref{eq:measureiteration},
\emph{i.e.}, if $\mu P = \mu$.
The existence of attractive invariant measures is intricately linked to
ergodic properties of the system.

\subsection{Invariant Measures and Ergodicity}

For the systems of interest to us a particular type of Markov chains are
of interest: the so-called iterated function systems (IFS). In an iterated
function system we are given a set of maps $\{ f_j : \Sigma \to \Sigma
\vert j \in {\cal J} \}$, where ${\cal J}$ is an index set. Associated to
these maps there are probability functions $p_j: \Sigma \to [0,1]$ such
that
\begin{equation*}
    X(k+1) = f_j(X(k)) \quad \text {with probability } \quad p_j(X(k)).
\end{equation*}
It is, of course, required that $\sum_{j\in {\cal J}} p_j(x) = 1$ for
all $x \in \Sigma$.
Sufficient conditions for the existence of a unique attractive
invariant measure can be given in terms of
the central notion of ``average contractivity'', which
can be traced back to
\cite{elton1987ergodic,BarnsleyDemkoEltonEtAl1988,barnsley1989recurrent}:
\begin{thm}[Barnsley et al. \cite{BarnsleyDemkoEltonEtAl1988}]
\label{Barnsley}
  Let $\Sigma \subset \R^n$ be closed.
  Consider an IFS with a finite index set ${\cal J}$,  and Lipschitz maps
  $f_j:\Sigma \to \Sigma$, $j\in {\cal J}$.
  Assume that the probability functions $p_j$ are Lipschitz continuous and
  bounded below by $\eta > 0$.
  If there exists a $\delta>0$  such that for all $x, y \in \Sigma, x \not = y$
  \begin{align*}
    \sum_{j\in {\cal J}} p_j(x) \log \left( \frac{ \|{f_j (x) - f_j (y)}\| }{ \|x-y\| } \right) < -\delta < 0,
  \end{align*}
  then there exists an attractive (and hence unique) invariant measure
  $\mu$ for the IFS.
\end{thm}

We can combine Theorem~\ref{Barnsley} with an ergodic theorem by Elton,
\cite{elton1987ergodic}, to obtain that for all (deterministic) initial conditions $x \in
\Sigma$ the limit
\begin{equation}
\label{eq:ergodicprop}
   \lim_{k \to\infty} \frac{1}{k+1} \sum_{\nu=0}^k  X(\nu) = \mathbb{E}(\mu)
\end{equation}
exists almost surely ($\prob_{x_0}$)  and is
independent of $x_0 \in \Sigma$. The limit is given by the expectation
with respect to the invariant measure $\mu$.
For more general theorems,
the reader is referred to  recent surveys \cite{iosifescu2009iterated,stenflo2012survey}. For one of the first uses of IFS in the control community, see \cite{Branicky1994,Branicky1998}.\newline

More generally, an invariant measure $\mu$ is called ergodic for a Markov
process, if for $\mu$-almost all initial conditions \eqref{eq:ergodicprop}
holds almost surely.

\begin{rem}
    From the point of view of applications in smart cities, such an
    ergodic property should be a minimum requirement. We want to avoid
    situations, where the average allocation of resources to agents depends
    on their initial conditions, on possible initial conditions of
    controllers and filters, etc. In addition, it is desirable to shape
    the expected value so that an overall optimum is obtained.
\end{rem}

With this background, our general problem considered in this paper is modelled as a Markov chain on a state space
representing all the system components. We thus let $X_S = \{ (x_i) \}$
be the set of vectors representing the possible values for the agents. The
spaces $X_F,X_C$ contain the possible internal states for filter and
central controller. Our system thus evolves on the state space $\Sigma :=
X_S \times X_C \times X_F$.


\subsection{Ergodic Invariant Measures and Coupling}

An important observation is that coupling arguments provide
 criteria for the non-existence of a unique invariant measure.
Coupling arguments have been used since the theorem of Harris \cite{harris1960lower,lindvall2012lectures}, and
are hence sometimes known as Harris-type theorems.
Generally, they link the existence of a coupling with
the forgetfulness of initial conditions.


To formalise these notions, let us denote the
the space of trajectories of a $\Sigma$-valued Markov chain $\{ X(k)\}_{k\in\N}$,
\emph{i.e.}, the space of all
sequences $(x(0),x(1),x(2),\ldots)$ with $x(k) \in \Sigma$, ${k\in\N}$,
by $\Sigma^\infty$ (the ``path space'').
The measure space over $X^\infty$
is denoted by $M(\Sigma^\infty)$.
Recall, for example, $\prob_\lambda \in M(\Sigma^\infty)$,
the probability measure induced on
the path space by the initial distribution $\lambda$ of $X_0$.

A coupling of two measures $P_{\mu_1},P_{\mu_2} \in M(\Sigma^\infty)$ is a measure on $\Sigma^\infty\times \Sigma^\infty$ whose marginals coincide with $P_{\mu_1}, P_{\mu_2}$.
To be precise,
consider $\Gamma \in M(\Sigma^\infty\times \Sigma^\infty)$, i.e., a measure over the product of the two path spaces.
Clearly, $\Gamma$ can be projected to the space of measures over one or the other path space $\Sigma^\infty$; we denote the projectors $\Pi^{(1)} \Gamma$ and $\Pi^{(2)} \Gamma$.
The set $C(P_{\mu_1}, P_{\mu_1})$  of couplings
of $P_{\mu_1}, P_{\mu_2} \in M(\Sigma^\infty)$ is then defined by
\begin{align*}
    \{
\Gamma \in M(\Sigma^\infty \times \Sigma^\infty) \; : \;
\Pi^{(1)} \Gamma = P_{\mu_1}, \Pi^{(2)} \Gamma  = P_{\mu_2}
\}.
\end{align*}
We say that a coupling $\Gamma$ is an asymptotic coupling if $\Gamma$ has
full measure on the pairs of convergent sequences. To make this precise
consider the following set denoted ${\cal D}$:
\begin{align*}
    \left \{ (x_1, x_2) \in \Sigma^\infty \times \Sigma^\infty \; : \; \lim_{k\to \infty} \left\| x_1(k) - x_2(k) \right\| =0
    \right \}
\end{align*}
$\Gamma$ is an asymptotic coupling if $\Gamma({\cal D}) = 1$.
The following statement if a specialization of \cite[Theorem 1.1]{hairer2011asymptotic}
to our situation:

\begin{thm}[Hairer et al. \cite{hairer2011asymptotic}]
\label{coupling-argument}
Let $P$ be a Markov operator admitting two ergodic invariant measures $\mu_1$ and $\mu_2$. The following are equivalent:
\begin{enumerate}
\item[(i)] $\mu_1 = \mu_2$.
\item[(ii)] There exists an asymptotic coupling of $P_{\mu_1}$ and $P_{\mu_2}$.
\end{enumerate}
\end{thm}

Consequently, if no asymptotic coupling of $P_{\mu_1}$ and $P_{\mu_2}$ exists,
then $\mu_1$ and $\mu_2$ are distinct.


\section{The Existence of a Unique Invariant Measure}

Let us consider the very simple setting of Example \ref{linear}
and show that the unique invariant measure exists:

\begin{thm} \label{thm01}
Consider the feedback system depicted in Figure \ref{system}, with ${\mathcal C}$ and ${\mathcal F}$ given in (\ref{eq_c}) and (\ref{eq_f}). 
Assume that each agent  
 $i \in \{1,\cdots,N\}$ has state $x_i(k)$ governed by the following affine stochastic difference equation:
\begin{equation}
x_i(k+1) = w_{ij}(x_i),
\end{equation}
where the affine mapping $w_{ij}$ is chosen at
each step of time according to a Dini-continuous
probability function $p_{ij}(x_i, q(k))$, out of $w_{ij}(x_i) \defeq A_i x_i + b_{ij},$
where $A_i$ is a Schur matrix and for all $i$, $q(k)$, $\sum_j p_{ij}(x_i, q(k)) = 1$. In addition,
suppose that there exist scalars $\delta_i > 0$ such that $p_{ij}(x_i,\pi)
\geq \delta_i > 0$; that is, the probabilities are bounded away from
zero. Then, for every stable linear controller $\mathcal{C}$ and every
stable linear filter $\mathcal{F}$, the feedback loop converges in
distribution to a unique invariant measure.
\end{thm}

\begin{proof}
Following \cite{BarnsleyDemkoEltonEtAl1988}, the proof is centred at the construction of an iterated function system (IFS) with place-(state-)dependent probabilities that describes the feedback system. To this end, consider the augmented state $$\xi \defeq [x'~,~y~,~\tilde y~,~z_f'~,~\hat y~,~e~,~z_c'~,~q]'$$, whose dynamic behaviour is described by the difference equation
\begin{equation}
\label{dynamical}
\xi(k+1) = w_\ell(x) \defeq \mathcal{A}\xi(k) + b_\ell,
\end{equation}
where $\mathcal{A}$ is the matrix
\begin{equation}
\notag
\Scale[0.8]{
\begin{bmatrix}
\hat{A} &  &  &  &  &  &  & \\
\mathbf{1}'\hat A & 0 &  &  &  &  &  & \\
0 & J & L &  &  &  &  &  \\
0 & B^f & C^f & A^f &  &  &  &  \\
0 & D^f B^f & D^f C^f & D^f A^f & 0 & & & \\
0 & -D^f B^f & -D^f C^f & -D^f A^f & 0 & 0 & & \\
0 & 0 & 0 & 0 & 0 & B^c & A^c & & \\
0 & -D^c D^f B^f & - D^c D^f C^f & -D^c D^f A^f & 0 & C^c B^c & C^c A^c & 0
\end{bmatrix}},
\end{equation}
where $J$ and $L$ are from \eqref{JL}, $\hat A \defeq \mathbf{diag}(A_i)$,
and $b_\ell$ is built from all the combinations of the vectors $b_{ij}$
and other signals. To apply Corollary 2.3 from
\cite{BarnsleyDemkoEltonEtAl1988}, two observations must be made. First,
note that each map $w_\ell$ is chosen with probability $p_\ell(\xi) \geq
\prod_{i=1}^N \delta_i > 0$ and, thus, they are bounded away from
zero. Second, since $\sigma({\mathcal A}) = \sigma(\hat A) \cup \sigma(L)
\cup \sigma(A^f) \cup \sigma(A^c) \cup \{0\}$ and, by hypothesis, $A_i$,
$A^f$ and $A^c$ are Schur matrices, then for any induced matrix norm
$\|\cdot\|$ there exists an $m \in \N$ sufficiently large such that
$\|{\mathcal A}^m\| < 1$. This provides the desired contraction on average
(after a finite number of steps). The result then follows from a variant
of Theorem~\ref{Barnsley}, see \cite{BarnsleyDemkoEltonEtAl1988}. The
proof is complete.
\end{proof}

\parindent0mm Some comments on the result of Theorem \ref{thm01} are in order.

\begin{rem}
Dini's condition on the probabilities may, obviously, be replaced by simpler, more conservative assumptions, such as Lipschitz or H\"older conditions \cite{BarnsleyDemkoEltonEtAl1988}.
\end{rem}

\begin{rem}
As we shall see, the requirement $p_{ij}(x_i,\pi) \geq \delta_i > 0$ in the theorem statement is not an artefact of our analysis, as $p_{ij}(x_i, \pi) = 0$ may lead to a non-ergodic behaviour.
\end{rem}

We should also explicate:

\begin{prp}
Under the assumptions of Theorem~\ref{thm01}: 
The case of $x_i \in \{ 0, 1\}$ is obtained by setting the $A_i$ to zero
and by introducing $b_{i0}=0$ and $b_{1i}=1$.
\end{prp}

\begin{rem}Finally, considering that Theorem \ref{Barnsley} does not require linearity,
it is clear that one can extend the results to non-linear systems, as we do
in a follow-up of this paper.
\end{rem}

\section{Switched Controllers}

Next, let us consider negative results.
Switched controllers are widely adopted by designers in practical applications that span several areas of engineering \cite{Shorten07}. The gain in flexibility provided by these controllers makes them suitable candidates for some smart cities problems.
In this section, consider the following structure for the controller ${\mathcal C}$
\begin{equation} 
{\mathcal C} ~:~ \left\{ \begin{array}{rcl}
x_c(k+1) & = & A_{c\sigma(k)} x_c(k) + B_{c\sigma(k)} e(k), \vspace{0.1cm} \\
\pi(k) & = & C_{c\sigma(k)} x_c(k) + D_{c\sigma(k)} e(k),
\end{array} \right.
\end{equation} 
where, once again, $x_c\in \R^{n_c}$ is its internal state and $\sigma \,
: \, \N \to [n_s] \defeq \{1,\cdots,n_s\}$ is the {\em switching signal},
which, at each time step, assigns one of the $n_s$ constituent systems $(A_{ci}, B_{ci}, C_{ci}, D_{ci})$, $i \in [n_s]$. Two complementary behaviours for the switching function are usually considered whenever one is analysing switched systems: switching control and switching perturbation.

\begin{eg}[Switched controller I]
Let us consider a simple example consisting of two agents ${\mathcal S}_1$ and ${\mathcal S}_2$ with only one common resource that must be shared, which means $r = 1$. As before, let us denote by $x_1$ and $x_2$ the agents states, which indicate whether each agent has access to the common resource or not; that is, $x_i \in \{0,1\}$, for $i \in \{1,2\}$. For simplicity, consider ${\mathcal F} = 1$ and, thus, $\hat{y} = y = x_1 + x_2$. Once again, each agent's behaviour is affected by the broadcast pricing signal $\pi$, which is computed by the controller ${\mathcal C}$. Let us assume each agent ${\mathcal S}_i$ switches its current state $x_i$ with probability
\begin{equation} \label{prb}
\mathbb{P}(x_i(k+1) = j ~|~ x_i(k) \neq j) = \pi,
\end{equation}
for all $k \in \N$. Given the identity above, in an attempt to adequately control the agents stochastic behaviour, the switched controller
\begin{equation}
  {\mathcal C} ~ : ~ \pi(k) = \frac{1}{2} |e(k)|, ~ \forall k \in \N.
\end{equation}
seems to be a suitable candidate. Under the action of this controller, the
closed-loop system depicted in Figure \ref{system} corresponds to the
Markovian process whose transitions are represented in Figure
\ref{markov_ne}, where the chain modes are $(x_1,x_2)$. As this Markov
chain presents two different stationary states, the process is not
ergodic; see \cite{Leon_Garcia}. Therefore, under mild conditions, it may
happen that $x_1$ will never get access to the resource.

\ifpdf

\begin{figure}
\centering
\includegraphics[width=0.75\columnwidth,clip=true, trim=0cm 0cm 0cm 0cm]{transitions}
\caption{Closed-loop transitions.}\label{markov_ne}
\end{figure}

\else

\begin{figure}
  \centering
  \begin{psfrags}
    \psfrag{00}[c]{\footnotesize $~(0,0)$}
    \psfrag{01}[c]{\footnotesize $~(0,1)$}
    \psfrag{10}[c]{\footnotesize $~(1,0)$}
    \psfrag{11}[c]{\footnotesize $~(1,1)$}
    \psfrag{q}[c]{\footnotesize $\frac{1}{4}$}
    \psfrag{p}[c]{\footnotesize $1$}
  \includegraphics[width=0.6\columnwidth]{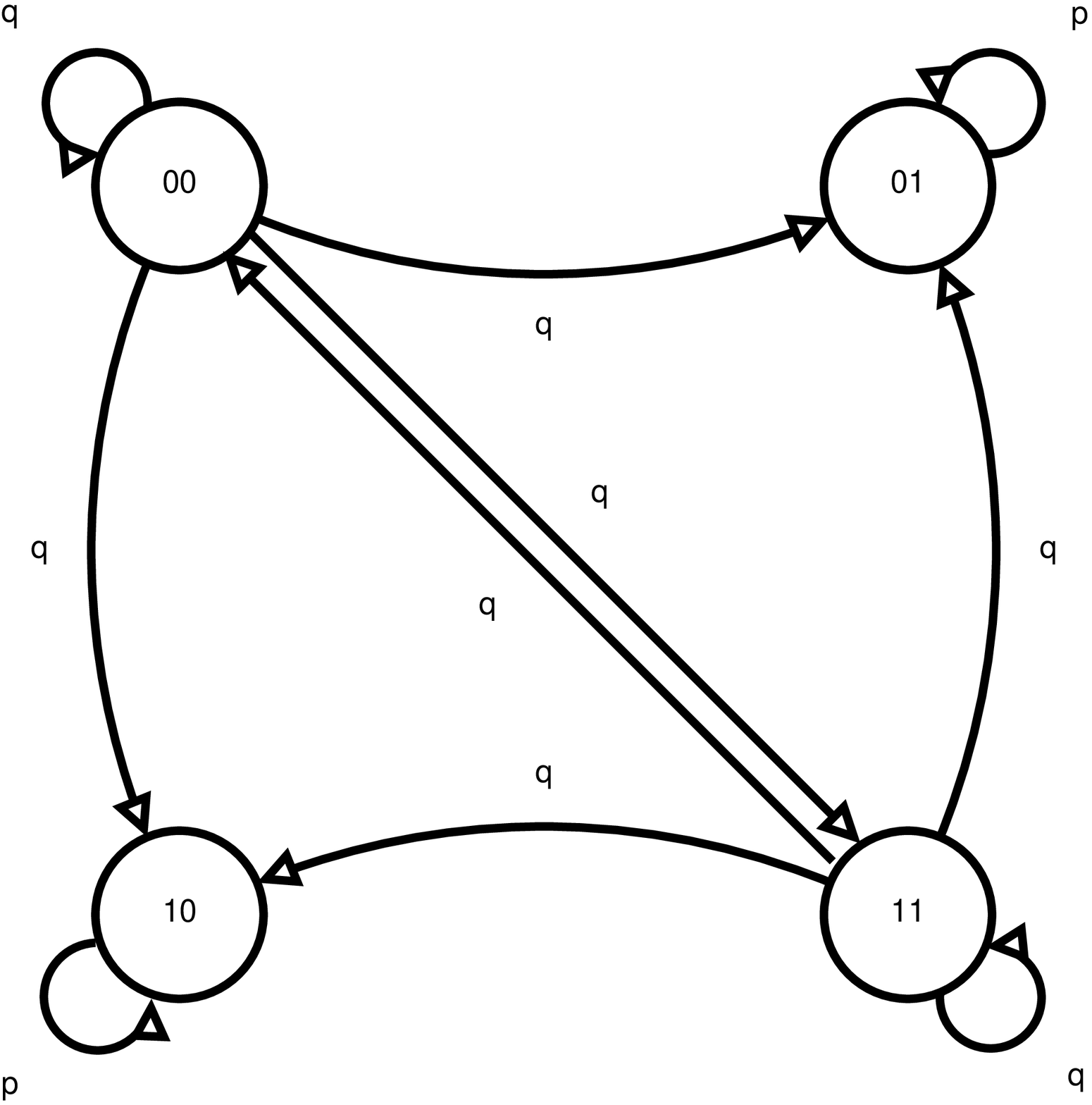}
  \caption{Closed-loop transitions.}\label{markov_ne}
  \end{psfrags}
\end{figure}

\fi

\end{eg}

\begin{eg}[Switched controller II]
Inspired by the previous example, let us now consider a more comprehensive
case. Consider the problem where $r = 50$ units of a certain resource must
be shared among $N = 100$ agents. As before, each agent ${\mathcal S}_i$
has a state $x_i \in \{0,1\}$, with switching law given by
(\ref{prb}). The broadcast signal is $\pi$, computed by the controller
${\mathcal C}$. Once again we make use of an input-dependent switched
controller without memory
\begin{equation}
{\mathcal C} ~:~ \pi(k) = \frac{1}{100}|e(k)|,\quad k \in \N.
\end{equation}
For each initial number of active agents between $0$ and $N$, we performed
Monte-Carlo experiments with $10^4$ realisations. In Figure
\ref{markov_ne}, the results for two complementary situations are
shown. The upper sequence represents the expected value of resource
consumption for agents whose initial state is $x_i(0) = 1$,  whereas the
lower series corresponds to the same value for agents with $x_i(0) =
0$. This plot shows that there exists a dependence of the expected value
$\bar x_i$ on the initial condition of an agent, implying the process is non-ergodic.

From a control-theoretical viewpoint this result may seem adequate, since
$\lim_{k
\to\infty}  e(k) = 0$ almost surely. Thus, resources are totally distributed. However, as the obtained closed-loop system does not present ergodicity, the resources are not shared fairly amongst all agents. Hence, the controller fails to achieve predictability.

\ifpdf

\begin{figure}
  \centering
  \includegraphics[width=\columnwidth,clip=true, trim=2cm 2cm 2cm 15cm]{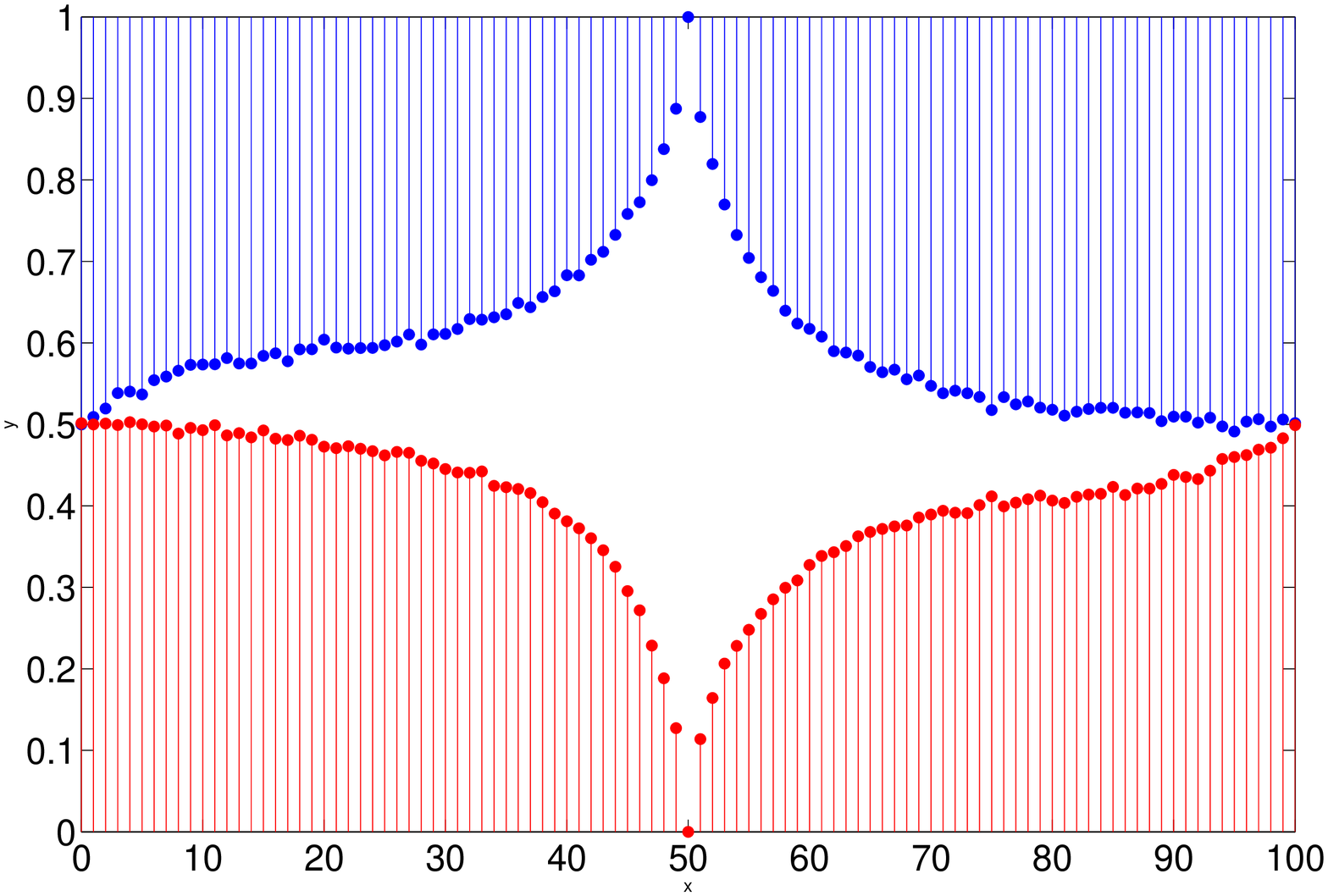}
  \caption{Switched controller simulation results:
average resource consumption
against the number of initial active systems}\label{nerg}
\end{figure}

\else

\begin{figure}
  \centering
  \begin{psfrags}
    \psfrag{x}[t]{\footnotesize Number of initial active systems}
    \psfrag{y}[b]{\footnotesize Average resource consumption}
  \includegraphics[width=\columnwidth]{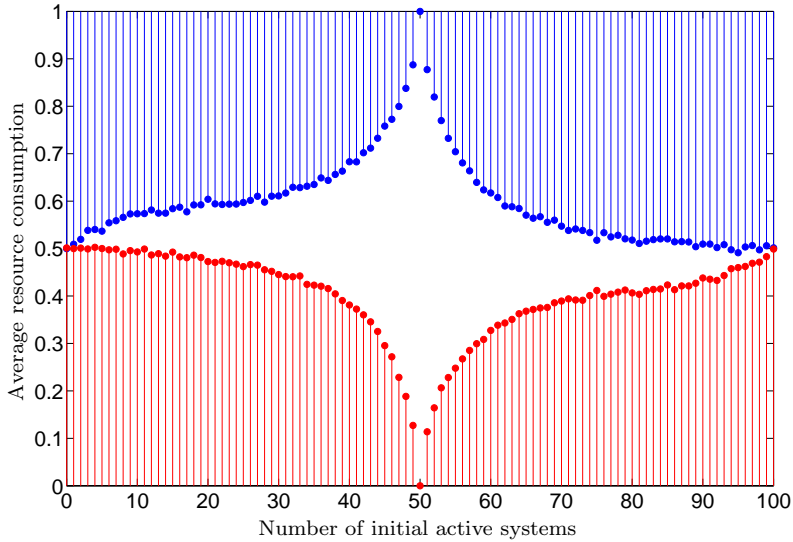}
  \caption{Switched controller simulation results.}\label{nerg}
  \end{psfrags}
\end{figure}

\fi
\end{eg}


We now consider the switching control case, in which the switching signal $\sigma$ is a design variable, together with the controller matrices $(A_{ci}, B_{ci}, C_{ci}, D_{ci})$, $i \in  [n_s]$, as in Example 1. A {\em state dependent} switching rule is a well-established structure for this control design problem \cite{Liberzon99}. In this setting, $\sigma$ is a composition of a function $u \, : \, \R^{n_c} \to  [n_s]$ with the controller state $x_c \, : \, \N \to \R^{n_c}$, that is, $\sigma(k) = u(x_c(k))$ for all $k \in \N$. Let us illustrate an issue that could possibly arise in this formulation. Consider that $n_s = 2$ and that the function $u$ divides the state space in two disjoint sets $\mathbb{S}_1$ and $\mathbb{S}_2$ (e.g. two half-spaces), such that $u(x) = 1$ for all $x \in \mathbb{S}_1$ and $u(x) = 2$ for all $x \in \mathbb{S}_2$. Since the dynamic behaviour of (\ref{eq_c}) has to be considered,
the mappings $w_\ell(\cdot)$ will also depend on the internal mode of the
switched controller ${\mathcal C}$ and, specifically, on the switching
function $\sigma$. Noting that, whenever $x_c(k)$ is in $\mathbb{S}_1$,
the switching function is such that $\sigma(k) = 1$, there is a null
probability of switching to the other controller mode, one concludes that
the assumptions in \cite{BarnsleyDemkoEltonEtAl1988} are not satisfied by
this controller. This issue can be circumvented if the designer allows a
small probability of switching to every other mode at any point in the
state space. In the switching perturbation case, the designer focuses on
determining the controller matrices $(A_{ci}, B_{ci}, C_{ci}, D_{ci})$, $i
\in  [n_s]$, aiming to provide a robust and stable closed-loop
system with respect to any possible switching signal. It can be seen as a
two-level design, where the controller realisation and the switching
signal can be devised separately, targeting different objectives. For this
particular case, we guarantee the existence of a unique invariant measure
by ensuring the existence of a contraction property, and under some
additional assumptions. For example, a contraction can be ensured using
the following lemma, which is based on the classic stability result
\cite{Daafouz02}. In the following statement, for symmetric matrices,
$P,Q$ we write $P\prec Q$, if $Q-P$ is positive definite.

\begin{lem} \label{lem01}
Consider the switched linear system
\begin{equation}
\label{eq:switchsys}
 x(k+1) = A_{\sigma(k)} x(k),
\end{equation}
where $x \in\R^{n}$ is the state and $\sigma \, : \, \N \to  [n_s]$ is a switching sequence. If there exist positive definite symmetric matrices $P_1,\cdots,P_{n_s} \in \mathbb{S}_+^{n}$ satisfying the following linear matrix inequalities
\begin{equation}\label{jamal}
A_i' P_j A_i - P_i \prec 0,
\end{equation}
for all $(i,j) \in [n_s]^2$, then $\eqref{eq:switchsys}$ is
exponentially stable for arbitrary switching sequences $\sigma$ and there exists $m \in \N$ sufficiently large such
that
\begin{equation}\label{norm}
\|A_{i_m} \cdots A_{i_1}\| < 1
\end{equation}
holds for any sequence of indices $i_1,\cdots,i_m \in  [n_s]$.
\end{lem}

\begin{proof}
First, note that, if there exist matrices $P_1,\cdots,P_{n_s}$ satisfying (\ref{jamal}), then there exists a sufficiently small scalar $\epsilon \in (0,1)$ such that
\begin{equation}\label{eps}
A_i' P_j A_i \prec (1 - \epsilon)^2 P_i
\end{equation}
hold for all $(i,j) \in  [n_s]^2$. Now define the quadratic,
time-varying function $v \, : \, \N\times \R^n  \to \R_+$, given by
\begin{equation}\label{lyap}
v(k,x) = x' P_{\sigma(k)} x.
\end{equation}
It is clear \cite{Horn1} that there exist constants $\alpha$ and $\beta$ such that
\begin{equation}\label{bound}
\alpha \|x\|^2 \leq x' P_i x \leq \beta \|x\|^2 \quad \forall \, i \in [n_s]
\end{equation}
hold for any $x \in \R^n$; here, $\|\cdot\|$ can be any vector norm.

Let us first prove that \eqref{eq:switchsys} is exponentially stable for
arbitrary switching signals $\sigma$. To this end, fix a switching
sequence $\sigma$ and
 note that (\ref{eps}) implies that $v$ satisfies along trajectories
$\{ x(k) \}$ corresponding to this switching sequence that
\begin{equation} \label{dec}
v(k+1,x(k+1)) \leq (1 - \epsilon)^2 v(k,x(k))
\end{equation}
for all $k \in \N$. Therefore, an inductive argument applied to (\ref{dec}) yields
\begin{equation}
v(k,x(k)) \leq (1 - \epsilon)^{2k} v(0,x(0)),
\end{equation}
for any $k \in \N$ and any given initial condition $x(0)\in \R^n$. Using the bounds (\ref{bound}), it follows that
\begin{equation} \label{exp}
\|x(k)\| \leq c (1 - \epsilon)^k \|x(0)\|
\end{equation}
holds for all $k \in \N$, where $c \defeq \sqrt{\beta/\alpha}$. As
$\sigma$ and $x(0)$ were arbitrary in this argument, we obtain exponential stability.

Finally, let us move our attention  to (\ref{norm}). Noting that
\begin{equation}
\|A_{i_k} \cdots A_{i_1}\| = \max_{w \neq 0}\frac{ \|A_{i_k} \cdots A_{i_1} w\| } {\|w\|},
\end{equation}
one can take $w$ as any initial condition $x(0)$ to $\Sigma$ and, since (\ref{exp}) holds for any switching sequence $\sigma$ and any $x(0)$, it follows that
\begin{equation}
\|A_{i_k} \cdots A_{i_1}\| = \max_{x(0) \neq 0} \frac{ \|x(k)\| }{\|x(0)\|} \leq c(1 - \epsilon)^k
\end{equation}
holds for all $k \in \N$. Since there exists a sufficiently large $m \in
\N$ such that $(1 - \epsilon)^m < c^{-1}$, (\ref{norm}) holds. This
completes the proof.
\end{proof}

\begin{rem}
The implication of Lemma \ref{lem01} is that, provided Equation \ref{eq:switchsys} represents the closed loop system, and some additional matrix inequality conditions hold, there always exists a sufficiently large number of jumps that ensures the contractivity of any possible switching chain. The existence of a unique invariant measure can be shown under assumptions similar to those of Theorems 1 and 3.  \end{rem}



\section{Control with poles on the unit circle}
\label{comments-pi}


We finally come to the most interesting observation of the paper.
In many applications, one wishes the error $e = r - \hat y$ to be convergent, that is, $\lim_{k \to \infty} e(k) = 0$. Consequently, controllers with integral action, such as the Proportional-Integral (PI) controller, are widely adopted \cite{Franklin,Franklin_dig}.
See Example \ref{linear2}, which can be implemented as:
\begin{equation}\label{pid1}
\pi(k) = \pi(k-1) + \kappa \big[ e(k) - \alpha e(k-1) \big],
\end{equation}
which means its transfer function from $e$ to $\pi$ is given by
\begin{equation}\label{pid2}
C(z) \defeq \frac{ \hat \pi (z) }{ \hat e(z) } = \kappa\frac{1 - \alpha z^{-1}}{1 - z^{-1}}.
\end{equation}
Since this transfer function is not asymptotically stable, any associated realisation matrix $A_c$ will not be Schur.
Note that this is the case for any controller with any sort of integral action, \emph{i.e.}, pole at $z = 1$.\newline

\begin{thm}
    \label{thm:pole}
    Consider $N$ agents with states $x_i, i=1,\ldots,N$. Assume
    that there is an upper bound $M$ on the different values the agents
    can attain, \emph{i.e.}, for each $i$ we have $x_i \in {\cal A}_i =\{
    a_1,\ldots,a_M \}\subset \R$
    for a given set ${\cal A}_i$.\newline

    Consider the feedback system in Figure \ref{system}, where ${\mathcal
      F}\, : \, y \mapsto \hat y$ is a finite-memory moving-average (FIR)
    filter.  Let ${\cal C}_L$ be a linear marginally stable single-input
    single-output (SISO) system with a pole $s_1 = e^{q i\pi}$ on the unit
    circle where $q$ is a rational number. Assume the controller
    ${\mathcal C} \, : \, e \mapsto \pi$ is the cascade of ${\cal C}_L$
    and a continuous map ${\cal C}_p: \R \to [0,1]$, i.e. if $\hat \pi
    (k)$ is the output of ${\cal C}_L$ at time $k$, then the signal from
    the controller is $\pi(k) = {\cal C}_p(\hat \pi(k))$. Then the
    following holds.

    \begin{enumerate}
      \item[(i)] The set ${\cal O}_{\cal F}$ of possible output values of the filter ${\cal F}$ is finite.
      \item[(ii)] If the real additive group ${\cal E}$ generated by
        $\{ r - \hat y \mid \hat y \in {\cal O}_{\cal F} \}$ is discrete, then
        the closed-loop system cannot be ergodic.
    \end{enumerate}
\end{thm}

\begin{rem}
One implication of the theorem is to illustrate the separation of the classical
performance of the closed loop, and ergodic behaviour. It is perfectly possible
for the closed loop to perform its regulation function well, and an the same time
destroy the ergodic properties of the closed loop.
\end{rem}

\begin{proof}
(i)    By assumption, the states of the agents $x \in \R^N$ can only attain
    finitely many values. Consequently, the set of possible values of $y$
    is finite and thus also the set of possible outputs of the filter is
    finite, as it is just the moving average over a history of finite
    length.\newline

(ii) We denote by ${\cal E}$ the additive subgroup of $\R$
    generated by the filter outputs.
By (i), the set of possible inputs to the linear part of the controller is
finite at any time $k\in \N$. Let $(A,B,C)$ be a minimal realization of
the linear controller with $A \in \R^{n_c \times n_c}$, $B,C^T\in \R^{n_c}$. Without any loss of generality, assume that
\begin{equation*}
    A =
    \begin{bmatrix}
        Q & 0 \\ 0 & R
    \end{bmatrix},\quad B =
    \begin{bmatrix}
        B_1 \\ B_2
    \end{bmatrix} , \quad C =
    \begin{bmatrix}
        C_1 & C_2
    \end{bmatrix}.
\end{equation*}
Here $Q$ is equal to $1,-1$ or a $2 \times 2$ orthogonal matrix with the
eigenvalues $s_1$ and $\overline{s_1}$. The matrix $R$ is marginally Schur
stable. We will concentrate on the first (or first two) component of the state of
the controller, which we denote by $x^{(1)}$. Given an initial value
$x_0^{(1)}$ these states are given by
\begin{equation*}
    x^{(1)}(k) = Q^k x^{(1)}_0 + \sum_{\nu=0}^{k-1}Q^{k-\nu-1} B_1 e(\nu).
\end{equation*}
For some power $K$ we have by assumption that $Q^K = I_2$. Thus
$x^{(1)}(k)$ is an element of the set ${\cal Z}(x_0)$ given by
\begin{equation*}
\left \{ Q^k x^{(1)}(0) +  \sum_{\nu=0}^{K-1} J^\nu B_1 e_\nu \ \Big\vert
      \ k =
      0,\ldots,K-1 ,e_\nu \in {\cal E} \right\}.
\end{equation*}
By assumption, this set is discrete in $\R$ or $\R^2$, as the case may
be. The state space of the controller may thus be partitioned into the
uncountably many equivalence classes under the equivalence relation on
$\R^{n_c}$ given by $x \sim y$, if $y^{(1)} \in {\cal Z}(x)$. These are
invariant under the evolution of the Markov chain.

Ergodic invariant measures which are concentrated on different equivalence
classes clearly cannot couple asymptotically as the respective
trajectories remain a positive distance apart. By Theorem~\ref{coupling-argument} the
Markov chain cannot be ergodic.
\end{proof}

\begin{rem}
We note that the conditions for non-ergodicity are satisfied, in
particular, if the resource consumption of agents is always a rational
number and the coefficients of the FIR filter are also rational. For
implementations on standard computers this will always be the case.
Other numerical issues issues may arise, though.
\end{rem}


Let us illustrate the undesirable behaviour that may arise whenever a PI controller is being used in the closed-loop system.
In the first example, we point out that the integral action may be heavily dependent on the controller state initial condition. To this end, consider the feedback system in depicted in Figure \ref{system} with $N = 4$ agents, whose states $x_i$ are in the set $\{0,1\}$; as before, if $x_i = 1$, we say that agent $i$ has taken the resource or is {\em active}.\newline

Our main goal is to regulate the number of active agents around the
reference $r = 2$. We assume that two agents, namely $x_1$ and $x_2$, have the following probabilities of being active
\begin{equation}
p_{12}(x_i(k+1) = 1) = 0.02 +  \frac{0.95}{ 1 + \exp(-100(\pi(k) - 5))},
\notag
\end{equation}
whereas the remaining agents' probability of consuming the resource is given by
\begin{equation}
p_{34}(x_i(k+1) = 1) = 0.98 -  \frac{0.95}{ 1 + \exp(-100(\pi(k) - 1))}.
\notag
\end{equation}
Note that their behaviour is, thus, complementary; indeed, if the control signal $\pi(k) \gg 5$, then $x_1$ and $x_2$ tend to be active, and, on the other hand, if $\pi(k) \ll 1$, then $x_3$ and $x_4$  are more susceptible to take the resource.\newline

In this design problem, we implement two types of controllers ${\mathcal C}$: a PI controller and its lag approximant. For this example, the filter ${\mathcal F}$ is the moving avergage (FIR) filter defined in (\ref{fir}). The PI controller is the one implemented in (\ref{pid1}) with $\kappa = 0.1$ and $\alpha = -4$; this controller is approximated by a lag controller with $\kappa = 0.1$, $\alpha = -4.01$ and $\beta = 0.99$. Figure \ref{sim1} points out that the PI controller regulates the average number of active agents $\bar y$, whereas the lag controller presents a steady state error (as expected). However, Figure \ref{sim2} shows different average trajectories for the first agent $\bar x_1$ for different initial conditions of the controller ${\mathcal C}$, namely $x_c(0) = 50$ and $x_c(0) = -50$. As the figure points out, this agent's behaviour is completely dependent on the initial value of $x_c$ when ${\mathcal C}$ is the PI controller; this undesired behaviour vanishes on the long run when a lag controller is used -- that is, the system becomes ergodic. Figure \ref{Xc} illustrates the dynamic response of the controller state $x_c$ for both initial conditions and both controllers; both cases converge to the same value for the lag structure and this is not observed when the designer uses a PI.\newline

\begin{figure}
  \centering
  \begin{psfrags}
    \psfrag{a}[c]{\scriptsize $\quad x_c(0) = 50$}
    \psfrag{b}[l]{\scriptsize $x_c(0) = -50$}
    \psfrag{x}[t]{\footnotesize $k$}
    \psfrag{y}[b]{\footnotesize $\bar y(k)$}
    \psfrag{y2}[b]{\footnotesize $\bar x_1(k)$}
    \includegraphics[width=.7\columnwidth]{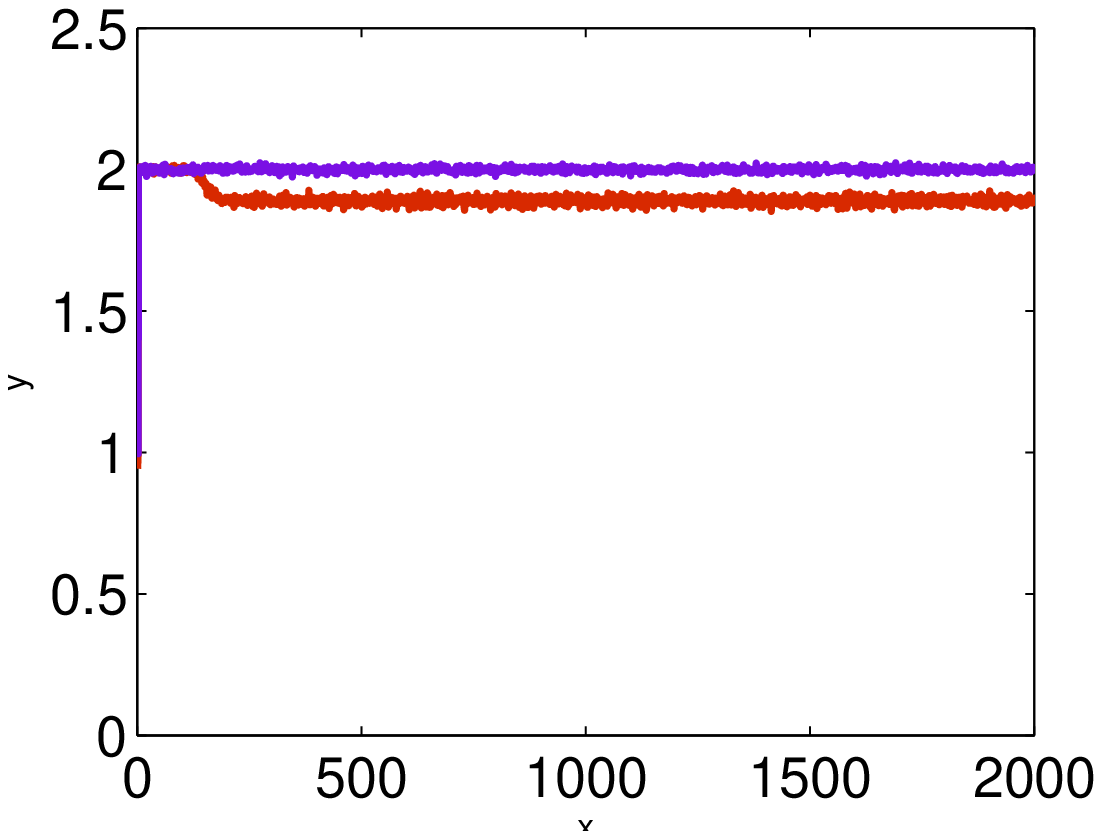}
  \caption{Average number of active agents dynamics.}\label{sim1}
  \end{psfrags}
\end{figure}

\begin{figure}
  \centering
  \begin{psfrags}
    \psfrag{a}[c]{\scriptsize $\quad x_c(0) = 50$}
    \psfrag{b}[l]{\scriptsize $x_c(0) = -50$}
    \psfrag{x}[t]{\footnotesize $k$}
    \psfrag{y}[b]{\footnotesize $\bar y(k)$}
    \psfrag{y2}[b]{\footnotesize $\bar x_1(k)$}
    \includegraphics[width=.7\columnwidth]{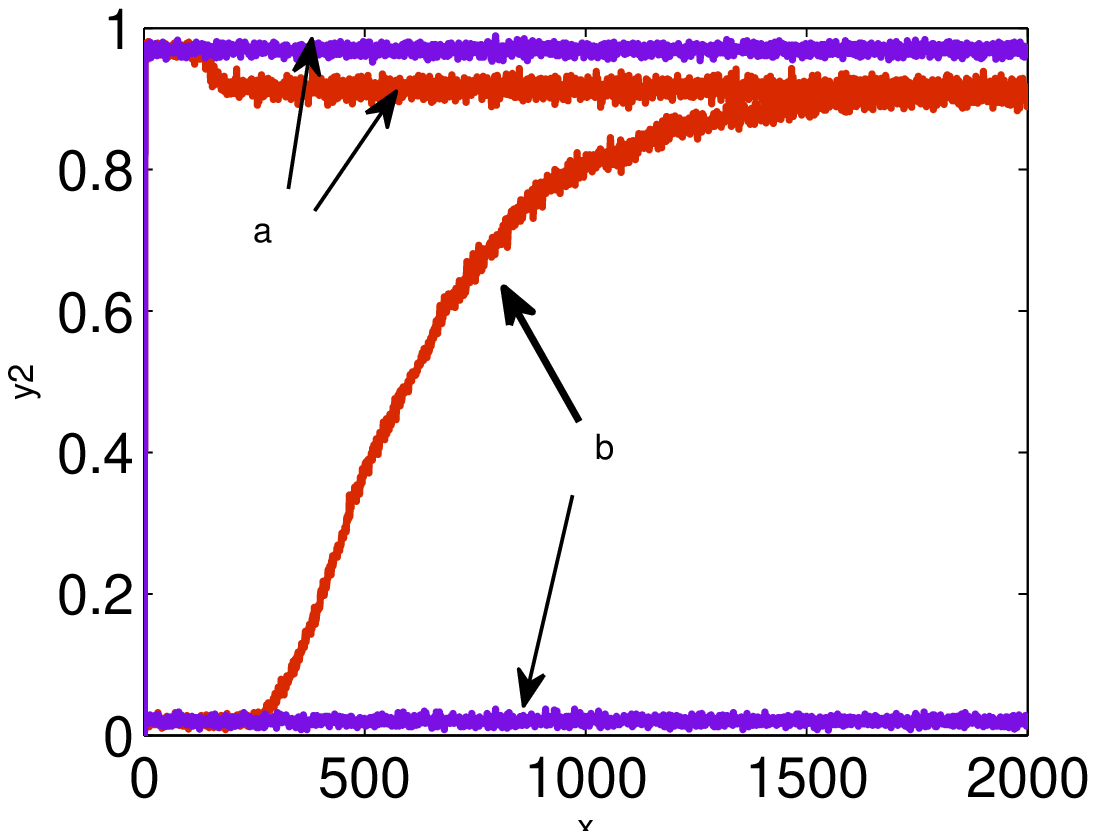}
  \caption{Average trajectory of the first agent.}\label{sim2}
  \end{psfrags}
\end{figure}

\begin{figure}
  \centering
  \begin{psfrags}
    \psfrag{a}[l]{\footnotesize $x_c(0) = 50$}
    \psfrag{b}[l]{\footnotesize $x_c(0) = -50$}
    \psfrag{x}[t]{\footnotesize $k$}
    \psfrag{y}[b]{\footnotesize $\bar x_c(k)$}
    \includegraphics[width=0.9\columnwidth]{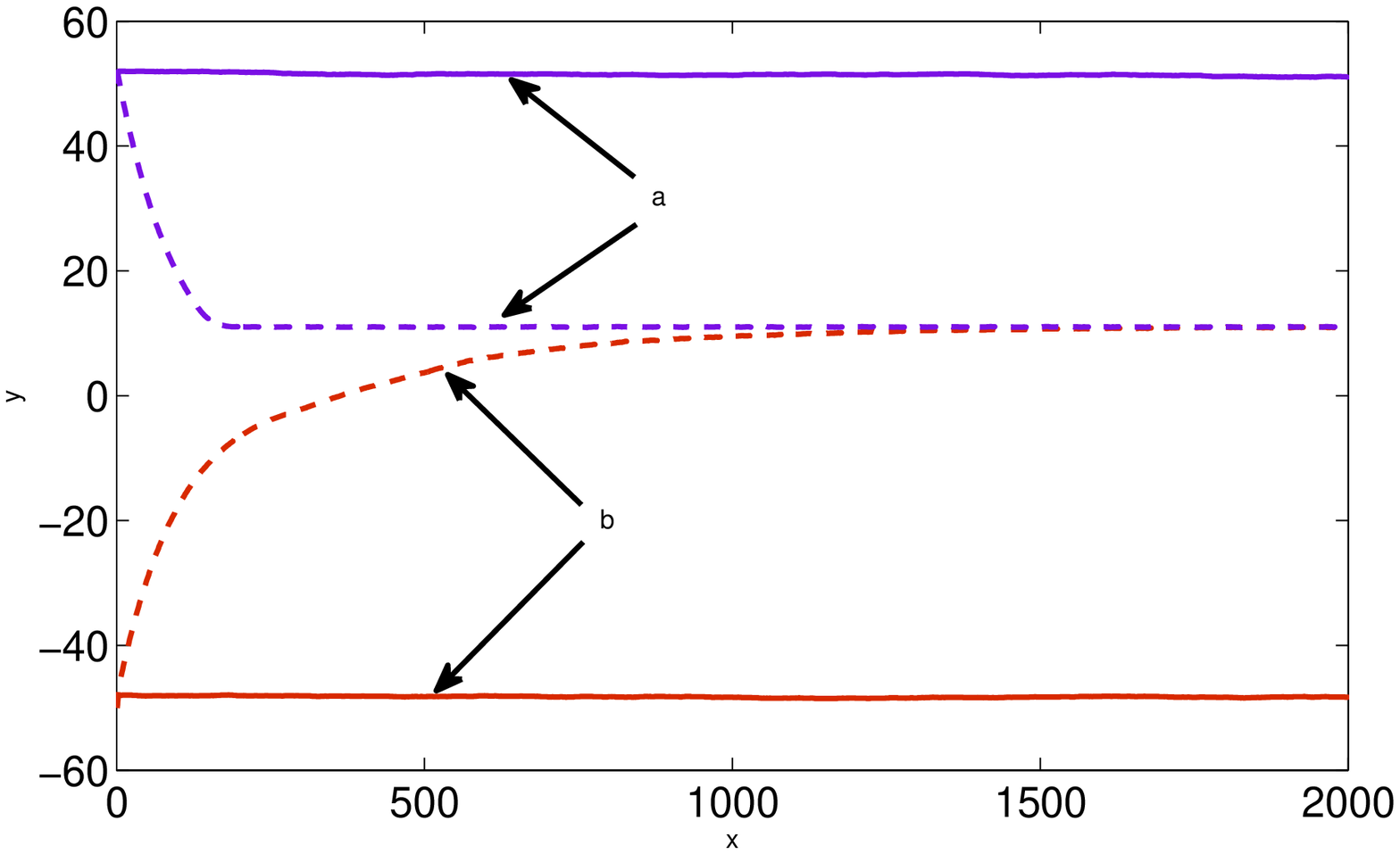}
  \caption{Average controller state dynamics.}\label{Xc}
  \end{psfrags}
\end{figure}

\section{Conclusions}

We have shown that across a number of problems studied under the banners of Smart Grid and Smart Cities,
classical proportional-integral (PI) control
cannot give rise to a feedback system with a unique invariant measure, which can lead to non-ergodic behaviour under benign conditions.
 The undesirable behaviour can be alleviated by the use of an IFS-based controller or by \emph{ad hoc} modifications to the classical controllers.
 Notice, in particular, that the input-dependent switched controllers fail to ensure the probabilities are bounded away from zero. This
  can be solved by replacing $f(\cdot) = \alpha|\cdot|$ by $\hat f(\cdot) = \alpha|\cdot| + \beta$, where both $\alpha$ and $\beta$ are defined to ensure $\pi \in (\epsilon,1 - \epsilon)$ for some $\epsilon > 0$. Similarly, the PI control can be modified by introducing a lead/lag compensator \cite{Franklin,Franklin_dig}, with the  effect of ensuring a contraction. The same effect can be achieved by adjusting the filter.
  Still, IFS-based controllers may present a principled approach to the problem.

\bibliographystyle{ieeetran}
\bibliography{ref,mdps}

\end{document}